\documentclass[12pt, reqno]{amsart}

\title{Expected multi-utility representations of preferences over lotteries}

\usepackage[T1]{fontenc}
\usepackage{amsmath}
\usepackage{amssymb}
\usepackage{amsthm}
\usepackage[left=3.5cm, right=3.5cm, paperheight=11.8in]{geometry}
\usepackage{hyperref}
\usepackage{fancyhdr}
\usepackage{enumitem}
\usepackage{comment}
\usepackage{nicefrac}
\usepackage{bm}
\usepackage{mathrsfs}
\usepackage{graphicx}
\usepackage[utf8]{inputenc}
\usepackage{cancel}
\usepackage{mathtools}
\usepackage{tikz}
\usetikzlibrary{decorations.pathreplacing,matrix,arrows,positioning,automata,shapes,shadows,calc,fadings,decorations,snakes,through,intersections}
\usepackage{float}

\AtBeginDocument{%
   \def\MR#1{}
}

\newtheorem{thm}{Theorem}[section]
\newtheorem{cor}[thm]{Corollary}
\newtheorem{lem}[thm]{Lemma}

\theoremstyle{definition} 
\newtheorem{defi}[thm]{Definition}
\let\olddefi\defi
\renewcommand{\defi}{\olddefi\normalfont}
\newtheorem{question}{Question}
\let\oldquestion\question
\renewcommand{\question}{\oldquestion\normalfont}
\newtheorem{example}[thm]{Example}
\let\oldexample\example
\renewcommand{\example}{\oldexample\normalfont}
\newtheorem{rmk}[thm]{Remark}
\let\oldrmk\rmk
\renewcommand{\rmk}{\oldrmk\normalfont}
\newtheorem{claim}{\textsc{Claim}}

\hypersetup{
    pdftitle={Multi-utility representations of preferences over lotteries},
    pdfauthor={Paolo Leonetti },
    pdfmenubar=false,
    pdffitwindow=true,
    pdfstartview=FitH,
    colorlinks=true,
    linkcolor=blue,
    citecolor=green,
    urlcolor=cyan
}

\providecommand{\MR}[1]{}

\providecommand{\MR}{\relax\ifhmode\unskip\space\fi MR }

\providecommand{\href}[2]{#2}

\makeatletter
\@namedef{subjclassname@1991}{JEL code}
\makeatother

\begin{document}

\author[P.~Leonetti]{Paolo Leonetti}
\address[P.~Leonetti]
{Department of Decision Sciences, Universit\`a Luigi Bocconi, via Roentgen 1, Milan 20136, Italy}
\email{leonetti.paolo@gmail.com}
\urladdr{\url{https://sites.google.com/site/leonettipaolo/}} 


\begin{abstract} 
\noindent 
Let $\succsim$ be a binary relation on the set of simple lotteries over a countable outcome set $Z$. We provide necessary and sufficient conditions on $\succsim$ to guarantee the existence of a set $U$ of von Neumann--Morgenstern utility functions $u: Z\to \mathbf{R}$ such that 
$$
p\succsim q 
\,\,\,\Longleftrightarrow\,\,\,
\mathbf{E}_p[u] \ge \mathbf{E}_q[u]
\,\text{ for all }u \in U
$$
for all simple lotteries $p,q$. 
In such case, the set $U$ is essentially unique. Then, we show that the analogue characterization does not hold if $Z$ is uncountable. 
This provides an answer to an open question posed by Dubra, Maccheroni, and Ok in 
[J. Econom. Theory~\textbf{115} (2004), no.~1, 118--133]. 
Lastly, we show that different continuity requirements on $\succsim$ allow for certain restrictions on the possible choices of the set $U$ of utility functions (e.g., all utility functions are bounded), providing a wide family of expected multi-utility representations. 
\end{abstract}
\maketitle
\thispagestyle{empty}

\section{Introduction}\label{sec:int}

Let $\succsim$ be a binary relation on a topological space $X$. 
The classical Debreu's utility representation theorem provides sufficient conditions on $\succsim$ and $X$ for the existence of a continuous utility function $u:X\to \mathbf{R}$ which represents $\succsim$, so that 
$$
x\succsim y 
\quad \Longleftrightarrow \quad 
u(x) \ge u(y),
$$
for all $x,y \in X$. In such case, $\succsim$ is necessarily complete and transitive. 

However, since the seminal works of Aumann \cite{MR174381} and Bewley \cite{MR1954892}, it is known that completeness may not be considered a basic trait of rationality of the decision maker. 
A long line of research studies this topic, 
see e.g. 
\cite{MR2956103, 
Dubra_et_al, 
MR2860226, 
MR3566440, 
MR1910069} and references therein. 
The main objective of this paper is to provide continuous multi-utility representation theorems for preference relations over lotteries which may fail to be complete. 
Thus, the basic question can be formulated as follows: Given binary relation $\succsim$ on a sufficiently well-behaved topological space $X$, we search for necessary and sufficient conditions on $\succsim$ for which there exists a collection $U$ of continuous utility functions $u: X \to \mathbf{R}$ such that 
$$
x\succsim y 
\quad \Longleftrightarrow \quad 
\forall u \in U, \,\,
u(x) \ge u(y),
$$
for all $x,y \in X$. 
Note that, in such case, the binary relation is necessarily reflexive and transitive, hence a preorder. 
In addition, 
we may interpret $\succsim$ as resulting from the unanimity of a family of complete preference relations, each one admitting a classical Debreu's continuous utility representation. 

In many instances, the underlying space $X$ has also the structure of (convex subset of a) vector space, so that it makes sense to require each continuous utility function $u$ to be, in addition, linear. 
To make an example, in the important work of Dubra, Maccheroni and Ok \cite{Dubra_et_al}, the authors provide necessary and sufficient conditions on a binary relation $\succsim$ over the set $\mathscr{P}$ of Borel probability measures on a compact metric space to satisfy the representation  
$$
p\succsim q 
\quad \Longleftrightarrow \quad 
\forall u \in U, \,\,
\mathbf{E}_p[u] \ge \mathbf{E}_q[u],
$$
for all $p,q \in \mathscr{P}$. Moreover, the set $U$ is essentially unique, in a precise sense. 
The interpretation of such representation is that the decision maker can show a lack of confidence in the evaluation of lotteries because he is unsure about his future tastes/risk attitudes, each represented by an utility function in $U$.

In the same work, the authors pose as open question in \cite[Remark 1]{Dubra_et_al} whether the analogue characterization holds replacing $\mathscr{P}$ with the set of simple lotteries on a given outcome set $Z$. Hereafter, a simple lottery (or, simply, a lottery) stands for a finitely-supported probability measure on $Z$. 
\emph{The main result of this work is that the answer is affirmative if and only if $Z$ is finite or countably infinite}. 
See Theorem \ref{thm:mainweaktopology} and Theorem \ref{thm:mainnegativecharacterization} below for the positive and negative case, respectively. Details follow in Section \ref{sec:intro}. 

Our methods of proofs will be sufficiently strong to provide a wide family of expected multi-utility representations. Indeed, we show that different continuity requirements on $\succsim$ allow for certain restrictions on the possible choices of the utility functions (e.g., all utility functions are bounded, or countably supported). Indeed, it is known that different continuity properties in an infinite dimensional framework may correspond to different behavioral implications, see e.g. \cite{Epstein}.

Lastly, removing also the transitivity and continuity requirements, we provide a representation for reflexive binary relations $\succsim$ which satisfy the independence axiom over lotteries on arbitrary outcome sets $Z$, in the same spirit of the coalitional representation given by Hara, Ok, and Riella in \cite[Theorem 1]{MR3957335}. 

\subsection{Motivations} 
Before we provide the details of our contribution, we remark that there are several economic reasons why one should be interested in the study of incomplete preference relations. Indeed, many authors strongly suggest, from both the normative and positive viewpoints, that preference relations should allow for potential indecisiveness of the decision makers, see e.g. \cite{MR174381, MR1954892}. Accordingly, within the revealed preference paradigm, it has been recently shown that incompleteness of preferences may arise due to status quo bias and the endowment effect, see e.g. \cite{EliazOK, MasOK}.
As an application, the resulting theory in \cite{EliazOK} has been able to cope with the classical preference reversal phenomenon.

In the case of our incomplete preferences, there are economic examples in which a decision maker is in fact composed of several agents, each one with a given utility function. For instance, in a social choice framework such as environmental policy, healthcare, or public policy, one commonly uses the Pareto dominance to rank alternatives (cf. First and Second Welfare Theorem), and such criterion is, of course, incomplete. As another real-world example, when purchasing a car, one needs to consider factors like price, fuel efficiency, safety, and style; here, the decision maker has to 
 compare preferences over these multiple criteria.

As anticipated before, we are going to study how to handle the problem of actually representing such preferences with expected multi-utility models as in \cite{Dubra_et_al}. It is worth to remark that the very same idea goes back at least to von Neumann and Morgenstern, see \cite[pp. 19--20]{vNM}, although without details. 
In this work, in particular, we focus on preferences over simple lotteries (i.e., probability distributions with finite support) on a given outcome set. There are many practical motivations which justify this choice (apart from answering an open question in \cite{Dubra_et_al}), and we list below several ones.

First, simple lotteries provide a way to model and analyze situations involving uncertainty and risk, cf. e.g. \cite{Bila, Drapeau}. Accordingly, they can be considered as the building blocks of utility theory, which is a cornerstone of microeconomics and rational decision making. As remarked by Evren in \cite{MR2398816}, simple lotteries 
can also be interpreted, in a game theory framework, as strategic decisions where players and the nature are confined to play simple strategies.

Second, in various fields such as finance, engineering, and operations research, problems involve decision variables that are inherently finite. Preferences over probability distributions on these variables are essential, e.g., for optimation problems. 
Simple lotteries are also employed in portfolio theory to assess the risk and return profiles of different investment portfolios. Diversification strategies, aimed at reducing risk, rely on preferences over (asset) lotteries. 
Similarly, 
in policy analysis and resource allocation problems, decision makers often have to make choices that affect finite populations or groups. Representations of these preferences help assess the impact of different policies on such finite groups. As a last example, in surveys or experiments, responses often take on discrete values, and the resulting probability distributions are finitely-supported. In this sense, studying preferences on the latter distributions allows decision makers to directly deal with the empirical data.

Third, simple lotteries are often easier to work with and understand than continuous ones. This simplicity can make it more feasible for individuals to express and communicate their preferences over such distributions. In addition, if preferences are not fully known or cannot be easily elicited, the study of incomplete preferences on probability measures offers a practical approach. Indeed, decision makers can express partial preferences over some aspects of a decision problem, even when they are unsure about others.

Lastly, 
even if representations of such incomplete preferences may seem rather technical, recent studies have shown that they are useful tools to understanding certain classes of \emph{complete} preferences. For example, the cautious expected utility model by Cerreia-Vioglio, Dillenberger, and Ortoleva \cite{Cerreia09} can be derived by applying the expected multi-utility representation of Dubra, Maccheroni, and Ok \cite{Dubra_et_al}. A similar remark holds for multi-prior expected multi-utility representations 
due to Galaabaatar and Karni \cite{Galaa}.


%
\section{Preliminaries and Main results}\label{sec:intro}

Let $\Delta$ be the set of finitely-supported probability measures on a given nonempty set $Z$, which is interpreted as the set of 
\emph{lotteries} over the set of possible outcomes $Z$, see e.g. \cite{Cerreia09, 
MR871146, Gilboa88, MR1106509, MR1906933}. 
A \emph{utility function} over $Z$ is a real-valued map $u: Z \to \mathbf{R}$. 
Note that $\Delta$ is a convex subset of the vector space of finitely-supported functions $p: Z\to \mathbf{R}$, hereafter denoted by $\mathbf{R}_0^Z$. 
Hence, a lottery is a nonnegative function $p \in \mathbf{R}_0^Z$ such that 
$\sum\nolimits_{z \in Z}p(z)=1$. 

\begin{defi}\label{def:preference}
Let $Z$ be a nonempty set. A binary relation $\succsim$ on $\Delta$ satisfies the 
\emph{independence axiom} if  
\begin{equation}\label{eq:indepednetdef}
p\succsim q \,\,\,
\Longleftrightarrow \,\,\,
\alpha p+(1-\alpha)r \succsim \alpha q+(1-\alpha)r
\end{equation}
for all lotteries $p,q,r \in \Delta$ and scalars $\alpha \in (0,1)$. 
\end{defi}

Hereafter, 
for each $p \in \mathbf{R}_0^Z$ and utility function $u: Z\to \mathbf{R}$, define
$$
\mathbf{E}_p[u]:=\sum\nolimits_{z \in Z}u(z)p(z).
$$
As anticipated in Section \ref{sec:int}, our main objective is to characterize binary relations $\succsim$ over lotteries through expected multi-utility representations in the sense of Dubra, Maccheroni, and Ok \cite{Dubra_et_al}, that is, to search for necessary and/or sufficient conditions on $\succsim$ for which there exists a set $U$ of utility functions (possibly, of a certain type) such that
\begin{equation}\label{eq:claimedmultiutility}
p\succsim q 
\,\,\, \Longleftrightarrow \,\,\,
\forall u \in U, \,\,
\mathbf{E}_p[u] \ge \mathbf{E}_q[u],
\end{equation}
for all lotteries $p,q \in \Delta$. 

To this aim, we endow the vector space $\mathbf{R}_0^Z$ with the weak-topology 
$$
w:=\sigma(\mathbf{R}_0^Z, \mathbf{R}^Z),
$$
so that a net $(p_i)_{i \in I}$ weak-converges to some $p \in \mathbf{R}_0^Z$ if and only if 
$$
\left(\hspace{.2mm}\mathbf{E}_{p_i}[u]\hspace{.2mm}\right)_{i \in I} \to \mathbf{E}_{p}[u]
\,\,\,
\text{ for all  }u \in \mathbf{R}^Z.
$$
Note that $(\mathbf{R}_0^Z, w)$ is a locally convex Hausdorff topological vector space. 
Moreover, we endow 
$\mathbf{R}^Z$ with the weak$^\star$-topology $\sigma(\mathbf{R}^Z, \mathbf{R}_0^Z)$ which is, equivalently, the topology of pointwise convergence, so that a net $(u_i)_{i \in I}$ weak$^\star$-converges to some $u\in \mathbf{R}^Z$ if and only if 
$$
(u_i(z))_{i \in I} \to u(z) 
\,\,\,
\text{ for all }z \in Z.
$$

We come now to the definition of sequential continuity of $\succsim$, as in \cite{Dubra_et_al}.
\begin{defi}\label{def:continuityaxiom}
Let $Z$ be a nonempty set. A binary relation $\succsim$ on $\mathbf{R}_0^Z$ is  $w$-\emph{sequentially continuous} if, for all weak-convergent sequences $(p_n)_{n\ge 1}$ and $(q_n)_{n\ge 1}$ with limits $p$ and $q$, respectively, then $p\succsim q$ whenever $p_n\succsim q_n$ for all $n \ge 1$.
\end{defi}

Finally, let $e \in \mathbf{R}^Z$ be the constant function $1$, and write $\overline{\mathrm{co}}$ for the closed convex hull operator and $\mathrm{cone}(U):=\{\lambda x: \lambda>0, x \in U\}$ 
for the cone generated by a nonempty set $U$ in a topological vector subspace of $\mathbf{R}^Z$ containing $e$, 
cf. Section \ref{sec:proofs} for details. 
Following the notation in \cite{Dubra_et_al}, 
we set also 
$$
\langle \, U\,\rangle:=
\overline{\mathrm{co}}(\mathrm{cone}(U)+\left\{\theta e: \theta \in \mathbf{R}\right\}).
$$
We can now provide an affirmative answer to the open question in \cite[Remark 1]{Dubra_et_al} for the case of countable sets.\footnote{In this work, \textquotedblleft countable\textquotedblright\ stands for \textquotedblleft finite or countably infinite.\textquotedblright} 

\begin{thm}\label{thm:mainweaktopology}
Let $Z$ be a nonempty countable set and $\succsim$ be a binary relation on $\Delta$. 
Then $\succsim$ is a $w$-sequentially continuous preorder which satisfies the independence axiom if and only if there exists a nonempty set $U$ of utility functions $u: Z\to \mathbf{R}$ such that \eqref{eq:claimedmultiutility} holds for all lotteries $p,q \in \Delta$. 

In addition, if $V\subseteq \mathbf{R}^Z$ is another nonempty set of utility functions with the same property, then 
$\langle \, U\,\rangle=\langle \, V\,\rangle$.
\end{thm}

It is worth to remark that the case where $Z$ is finite has been already characterized in \cite{Dubra_et_al}, replacing the $w$-sequential continuity of $\succsim$ with the, here equivalent, requirement 
$$
\{\alpha \in [0,1]: \alpha p+(1-\alpha)q \succsim \alpha r+(1-\alpha)s\} \text{ is closed}
$$
for all lotteries $p,q,r,s \in \Delta$, see \cite[Proposition 1]{Dubra_et_al}. 
Also, when paired with certain continuity requirements on the reflexive binary relation $\succsim$, the independence axiom can be slightly weakened, see e.g. \cite[Lemma 1]{Dubra_et_al}.

The first part of Theorem \ref{thm:mainweaktopology} has been obtained also by Cerreia-Vioglio in \cite{Cerreia09} and by McCarthy et al. \cite{MR4343799}, although with different proofs. 
In our case, the route is different. 
Indeed, we are going to provide in Corollary \ref{cor:necesuff} a necessary and sufficient condition on the set of outcomes $Z$ (with an arbitrary cardinality) for which every $w$-sequentially continuous preorder satisfying the independence axiom admits an expected multi-utility representation, and viceversa. 
At that point, the proof of Theorem \ref{thm:mainweaktopology} will follow by checking that such condition is satisfied by all countable sets. 

Next, we present our negative result that such condition does not hold for \emph{every} uncountable $Z$. This provides a negative answer to \cite[Remark 1]{Dubra_et_al} in the remaining cases.\footnote{A simpler negative result in this direction, replacing $w$-sequential continuity with  \textquotedblleft algebraic continuity\textquotedblright\ of $\succsim$, can be obtained from \cite[Example 4.3]{MR4343799}.} 
\begin{thm}\label{thm:mainnegativecharacterization}
Let $Z$ be an uncountable set. Then there exists a $w$-sequentially continuous preorder $\succsim$ on $\Delta$ which satisfies the independence axiom and such that the representation \eqref{eq:claimedmultiutility} fails for all nonempty sets $U$ of utility functions $u: Z\to \mathbf{R}$. 
\end{thm}


It is worth to remark that the above result does not contradict Dubra, Maccheroni and Ok's original characterization \cite{Dubra_et_al}. In the latter, as anticipated in Section \ref{sec:int}, the authors consider binary relations over the set $\mathscr{P}$ of Borel probability measures on a compact metric space $Z$, and replace the notion of $w$-sequential convergence $\lim_n p_n=p$ with the weaker one $\lim_n\mathbf{E}_{p_n}[u]=\mathbf{E}_p[u]$ for all \emph{continuous} $u: Z\to \mathbf{R}$ (which we denote hereafter by $C(Z)$-sequential convergence). In fact, a direct consequence 
from Theorem \ref{thm:mainnegativecharacterization} and \cite{Dubra_et_al} 
is the following: 
\begin{cor}\label{cor:noncontradiction}
Let $Z$ be an uncountable compact metric space. Then there exists a $w$-sequentially continuous preorder $\succsim$ on $\Delta$ which satisfies the independence axiom and such that: If $\succsim^\prime$ is a $C(Z)$-sequentially continuous preorder on $\mathscr{P}$ which satisfies the independence axiom and $p\succsim q$ implies $p\succsim^\prime q$ for all $p,q \in \Delta$, then there exist $p_0,q_0 \in \Delta$ such that $p_0\succsim q_0$ while $p_0\not\succsim^\prime q_0$. 
\end{cor}

The advantage of our approach is that the technical tools that we provide in Section \ref{sec:proofs} allow us to obtain analogue results in different settings. 
As an example, suppose that $Z$ is a nonempty countable set and define the weak-topology 
$$
\nu:=\sigma(\mathbf{R}_0^Z, \ell_\infty^Z),
$$
where $\ell_\infty^Z$ is the vector space of bounded utility functions (hence a net $(p_i)_{i \in I}$ in $\mathbf{R}_0^Z$ is $\nu$-convergent to $p\in \mathbf{R}_0^Z$ if and only if $(\mathbf{E}_{p_i}[u])_{i \in I}\to \mathbf{E}_p[u]$ for all $u \in \ell_\infty^Z$). The definition of $\nu$-sequential continuity for $\succsim$ goes as in Definition \ref{def:continuityaxiom} and, similarly, endow $\ell_\infty^Z$ with the weak$^\star$-topology $\sigma(\ell_\infty^Z, \mathbf{R}_0^Z)$. 

Then, we have the following bounded analogue of Theorem \ref{thm:mainweaktopology}.
\begin{thm}\label{thm:mainweaktopologybounded}
Let $Z$ be a nonempty countable set and $\succsim$ be a binary relation on $\Delta$. 
Then $\succsim$ is a $\nu$-sequentially continuous preorder which satisfies the independence axiom if and only if there exists a nonempty set $U$ of bounded utility functions $u \in \ell_\infty^Z$ such that \eqref{eq:claimedmultiutility} holds for all lotteries $p,q \in \Delta$. 

In addition, if $V\subseteq \ell_\infty^Z$ is another nonempty set of bounded utility functions with the same property, then $\langle\, U\,\rangle=\langle\, V\,\rangle$.
\end{thm}

A weaker version of the first part of Theorem \ref{thm:mainweaktopologybounded} has been obtained by Evren in \cite[Proposition 2]{MR2398816} by means of Banach space techniques. 

In both positive results (Theorem \ref{thm:mainweaktopology} and Theorem \ref{thm:mainweaktopologybounded}), 
the preorder $\succsim$ will be, in addition, monotone with respect to a given binary relation $\trianglerighteq$ on $Z$ if and only if the claimed set $U$ contains only increasing utility functions, see Corollary \ref{rmk:monotonicity} for details and definitions. 

It is worth noting that \emph{every} real vector space $X$ can be regarded as a $\mathbf{R}_0^Z$ space. 
Indeed, let $Z$ be a Hamel basis of $X$ and to consider the natural isomorphism $T: X\to \mathbf{R}_0^Z$ defined so that, for each vector $x \in X$, $T(x)(z)=\lambda_z$ 
for every 
$z \in Z$, where $\lambda_z$ is the real coefficient in the unique representation $x=\sum\nolimits_{z \in Z}\lambda_z z$. 

Lastly, let us characterize the binary relations $\succsim$ on $\Delta$ which are only reflexive and satisfy the independence axiom. In particular, $\succsim$ is not necessarily transitive nor (sequential) continuous. Analogous binary relations over Borel probability measures have been studied 
recently by Hara, Ok, and Riella in \cite{MR3957335} in the context of expected utility theory, cf. also \cite{MR2888839, 
MR3566440} for related results. 
\begin{thm}\label{thm:firstrepresentation}
Let $Z$ be a nonempty set and $\succsim$ be a binary relation over $\Delta$. Then $\succsim$ is reflexive and satisfies the independence axiom if and only if there exists a nonempty family 
$\mathscr{U}$ 
of nonempty sets $U$ of 
utility functions $u: Z \to \mathbf{R}$ such that 
\begin{equation}\label{eq:firstcharacterizationtoogeneral}
p\succsim q 
\quad \Longleftrightarrow \quad 
\forall U \in \mathscr{U}, \exists u \in U, \,\,
\mathbf{E}_p[u] \ge \mathbf{E}_q[u],
\end{equation}
for all lotteries $p,q \in \Delta$.
\end{thm}

Additionally, the family $\mathscr{U}$ 
in the above representation 
is essentially unique, in a precise sense, see Remark \ref{rmk:uniquenessgeneral} below. 

To prove our results, we proceed as follows. 
We collect in Section \ref{sec:proofs} the necessary preliminaries on dual pairs and representations of closed convex cones in topological vector spaces. 
In Section \ref{sec:coneslotteries}, we study the existence of expected multi-utility representations in the special case of cones $C\subseteq \mathbf{R}_0^Z$, for dual pairs of the type $(\mathbf{R}_0^Z, W)$. The proofs of Theorems \ref{thm:mainweaktopology}--\ref{thm:firstrepresentation} follow in Section \ref{sec:mainproofs}. 

Finally, we highlight below the main concepts and results that will appear in the sequel:
\begin{enumerate}[label={\rm (\alph{*})}]
\item Theorem \ref{thm:closedconesdecisiontheory} provides the relationship between expected multi-utility representations and [sequentially weak-]closed convex cones, together with the corresponding uniqueness of the representing set of utility functions. The statement is formulated in the abstract setting of dual pairs. The generality of the framework is justified by the fact the proofs of several important results (including the works of Dubra, Maccheroni, and Ok \cite{Dubra_et_al} and Ok and Weaver \cite{OkWeaver}) are implicitly using the very same strategy. 
\item Definition \ref{def:Lproperty} provides the new notion of vector subspaces $W\subseteq \mathbf{R}^Z$  of utility functions, that is, vector spaces with the $\mathrm{(L)}$-property. This will be interpreted as the ability of the decision maker to consider closeness of lotteries (with respect to $W$) in the same way for all their conditional lotteries. 
\item Corollary \ref{cor:necesuff} provides necessary and sufficient conditons on the property that \emph{every} sequentially continuous preorder on $\Delta$ satisfying the independence axiom admits an expected multi-utility representation as in \eqref{eq:claimedmultiutility}. The statement is formulated in the setting of the dual pairs $(\mathbf{R}_0^Z, W)$, where $W$ has the $\mathrm{(L)}$-property. Informally, a representation is always possible if and only if every sequentially closed convex cone is also closed. This will serve as the key tool to prove both our positive and negative main results.  
\item Theorem \ref{thm:mainintermediate} shows that there exists a notion of continuity for $\succsim$ such that, if $Z$ is an \emph{arbitrary} nonempty set, then $\succsim$ is a continuous preorder on $\Delta$ satisfying the independence axiom if and only if representation \eqref{eq:claimedmultiutility} holds for some set $U$ made of countably-supported utility functions. 
\item Corollary \ref{rmk:monotonicity} provides that analogue of our positive main result (with $Z$ countable) for \emph{monotone} binary relations. This will be equivalent to the existence of a representation \eqref{eq:claimedmultiutility} for some set $U$ made of increasing utility functions. 
\end{enumerate}

\section{Dual pairs}\label{sec:proofs}

Given a real topological vector space (tvs) $X$, we denote by $X^\star$ and $X^\prime$ its algebraic dual and topological dual, respectively. 
A nonempty subset $C\subseteq X$ is a \emph{cone} if $\lambda x\in C$ for all $x \in C$ and $\lambda > 0$, and it said to be pointed if $0 \in C$. 

A pair of vector spaces $(X,Y)$ is a \emph{dual pair} if there exists a bilinear map $\langle \cdot,\cdot \rangle: X\times Y\to \mathbf{R}$ such that the families $\{\langle \cdot, y\rangle: y \in Y\}$ and $\{\langle x, \cdot\rangle: x \in X\}$ separate the points of $X$ and $Y$, respectively. 
Unless otherwise stated, $X$ and $Y$ are endowed with the weak-topology $\sigma(X,Y)$ and the weak$^\star$-topology $\sigma(Y,X)$, respectively. 
Hence, for example, a net $(x_i)_{i \in I}$ in $X$ is weak-convergent to $x \in X$ if and only if the real net $\left(\langle x_i,y\rangle\right)_{i \in I}$ is convergent to $\langle x,y\rangle$ for each $y \in Y$. 
It is well known 
that both $X$ and $Y$ are locally convex Hausdorff tvs, and that $X^\prime=Y$ and $Y^\prime=X$, up to isomorphisms; 
conversely, if $X$ is a locally convex Hausdorff tvs, then $(X,X^\prime)$ is a dual pair, see e.g. \cite[Section 5.14]{MR2378491} . 

Given a dual pair $(X,Y)$, recall that if $S\subseteq X$ is a nonempty subset then its \emph{dual cone} is the weak$^\star$-closed convex cone
$$
S^\prime:=\{y \in Y: \forall x \in S, \, \langle x,y\rangle \ge 0\}.
$$ 
Also, it is known that every weak-closed convex cone $C\subseteq X$ is self-dual, meaning that $C=C^{\prime\prime}$; more explicitly, 
$$
C=\{x \in X: \forall y \in C^\prime, \,\,\langle x,y\rangle \ge 0\},
$$
see e.g. \cite[Theorem 2.13]{MR2317344}. The choice of $C^\prime$ is also unique in the following sense: 

\begin{thm}\label{thm:dualdualcone}
Let $(X,Y)$ be a dual pair and 
pick a weak-closed convex cone $C\subseteq X$ and a weak$^\star$-closed convex cone $D\subseteq Y$. 
Then $C=D^\prime$ if and only if $D=C^\prime$. 
\end{thm} 
\begin{proof}
See \cite[Theorem 3.8]{LeoPrin22}. 
\end{proof}

Hereafter, we use $\overline{\mathrm{co}}$ for the closed convex hull operator, where the closure is taken in the corresponding weak-topology of the underlying space, and $\mathrm{cone}(S)$ for the smallest cone containing $S\subseteq X$, that is, $\{\lambda x: \lambda>0, x \in S\}$. 

Analogously to Definition \ref{def:preference}, a binary relation $\succsim$ on a vector space $X$ satisfies the independence axiom if \eqref{eq:indepednetdef} holds for all $p,q,r \in X$ and $\alpha \in (0,1)$. 
\begin{defi}\label{defi:continuitypreferences}
Let $(X,Y)$ be a dual pair. A binary relation $\succsim$ on $X$ is \emph{continuous} if, for all weak-convergent nets $(p_i)_{i \in I}$ and $(q_i)_{i \in I}$ in $X$ with limits $p$ and $q$, respectively, then $p\succsim q$ whenever $p_i\succsim q_i$ for all $i \in I$. If the same property holds for sequences, in place of nets, then $\succsim$ is \emph{sequentially continuous}. 
\end{defi}

With these premises, we have the following representation of weak-closed convex cones. This is mainly a reference result, since most equivalences can be found in the literature (or used implicitly therein). 
\begin{thm}\label{thm:closedconesdecisiontheory}
Let $(X,Y)$ be a dual pair, $\succsim$ be a reflexive binary relation on $X$ which satisfies the independence axiom, and define 
\begin{equation}\label{eq:definitionC}
C:=\{\lambda (p-q): \lambda \ge 0, \,\, p,q \in X, \text{ and }\,p\succsim q\}.
\end{equation}
Consider the following properties\emph{:}
\begin{enumerate}[label={\rm (\roman{*})}]
\item \label{S1:item} $C$ is a weak-closed convex cone\emph{;}

\item \label{S2:item} $C$ is a $\tau$-closed convex cone for some consistent topology\footnote{Recall that two topologies on a vector space are said to be consistent if they induce the same topological duals.} $\tau$ on $X$\emph{;}

\item \label{S3:item} for each nonempty $U\subseteq Y$, 
$$
C=U^\prime \quad \Longleftrightarrow \quad 
\overline{\mathrm{co}}(\mathrm{cone}(U))=C^{\prime}\emph{;}
$$

\item \label{S4:item} $C=C^{\prime\prime}$\emph{;}

\item \label{S5:item} $C=U^\prime$ for some nonempty $U\subseteq Y$\emph{;}

\item \label{S6:item} for each nonempty $U\subseteq Y$, the equivalence 
$$
\forall p,q \in X, \quad \,\,\,
p\succsim q 
\,\,\,\Longleftrightarrow\,\,\,
\langle p,y\rangle \ge \langle q,y\rangle \text{ for all }y \in U
$$
holds if and only if $\overline{\mathrm{co}}(\mathrm{cone}(U))=C^{\prime}$\emph{;}

\item \label{S7:item} it holds
$$
\forall p,q \in X, \quad \,\,\,
p\succsim q 
\,\,\,\Longleftrightarrow\,\,\,
\langle p,y\rangle \ge \langle q,y\rangle \text{ for all }y \in C^\prime\emph{;}
$$

\item \label{S8:item} there exists a nonempty $U\subseteq Y$ such that
$$
\forall p,q \in X, \quad \,\,\,
p\succsim q 
\,\,\,\Longleftrightarrow\,\,\,
\langle p,y\rangle \ge \langle q,y\rangle \text{ for all }y \in U\emph{;}
$$

\item \label{S9:item} $\succsim$ is transitive and continuous\emph{;}

\item \label{S10:item} $\succsim$ is transitive and sequentially continuous\emph{;}

\item \label{S11:item} $C$ is a sequentially weak-closed convex cone\emph{.}
\end{enumerate}
Then 
the following implications hold\emph{:}

\begin{figure}[!htb]
\centering
\begin{tikzpicture}
[scale=2.2]
\node (a1) at (0,1){\ref{S7:item}};
\node (a2) at (0.705,0.705){\ref{S8:item}};
\node (a3) at (1,0){\ref{S1:item}};
\node (a4) at (0.705,-0.705){\ref{S2:item}};
\node (a5) at (0,-1){\ref{S3:item}};
\node (a6) at (-0.705,-0.705){\ref{S4:item}};
\node (a7) at (-1,0){\ref{S5:item}};
\node (a8) at (-0.705,0.705){\ref{S6:item}};

\node (b12) at (.35,1-.1){\rotatebox[origin=c]{-20}{$\Longleftrightarrow$}};
\node (b18) at (-.35,1-.1){\rotatebox[origin=c]{20}{$\Longleftrightarrow$}};
\node (b56) at (-.35,-1+.1){\rotatebox[origin=c]{-20}{$\Longleftrightarrow$}};
\node (b54) at (.35,-1+.1){\rotatebox[origin=c]{20}{$\Longleftrightarrow$}};
\node (b23) at (1-.1, .35){\rotatebox[origin=c]{-70}{$\Longleftrightarrow$}};
\node (b34) at (1-.1, -.35){\rotatebox[origin=c]{70}{$\Longleftrightarrow$}};
\node (b67) at (-1+.1, .35){\rotatebox[origin=c]{70}{$\Longleftrightarrow$}};
\node (b78) at (-1+.1, -.35){\rotatebox[origin=c]{-70}{$\Longleftrightarrow$}};

\node (a9) at (2.3,.7){\ref{S9:item}}; 
\node (a10) at (3.6,0){\ref{S10:item}}; 
\node (a11) at (2.3,-.6){\ref{S11:item}}; 

\draw[-implies,double equal sign distance] (a3) -- (a9);
\draw[-implies,double equal sign distance] (a9) -- (a10);
\draw[-implies,double equal sign distance] (a3) -- (a11);
\end{tikzpicture}
\end{figure}
\end{thm}
\begin{proof}
\ref{S1:item} $\Leftrightarrow$ \ref{S2:item}. By Mackey's theorem, see e.g. \cite[Theorem 8.9]{MR2317344}, all consistent topologies on $X$ share the same closed convex sets. 

\medskip

\ref{S1:item} $\Rightarrow$ \ref{S3:item}. Each section of the duality map $\langle \cdot,\cdot\rangle$ is continuous, hence $U^\prime=(\mathrm{cone}(U))^\prime=D^\prime$, where $D:=\overline{\mathrm{co}}(\mathrm{cone}(U))$ is a weak$^\star$-closed convex cone. The claim follows by Theorem \ref{thm:dualdualcone}.

\medskip 

\ref{S3:item} $\Rightarrow$
\ref{S4:item} $\Rightarrow$
\ref{S5:item} $\Rightarrow$ 
\ref{S1:item}. 
They are obvious.

\medskip

\ref{S3:item} $\Leftrightarrow$ \ref{S6:item}, 
\ref{S4:item} $\Leftrightarrow$ \ref{S7:item}, and 
\ref{S5:item} $\Leftrightarrow$ \ref{S8:item}. 
Considering that $\succsim$ is a reflexive binary relation which satisfies the independence axiom, these equivalences follow from the fact that $p\succsim q$ if and only if $p-q \in C$.

\medskip

To sum up, we have the equivalences \ref{S1:item} $\Leftrightarrow$ 
\ref{S2:item} $\Leftrightarrow$ 
\ref{S3:item} $\Leftrightarrow$ 
\ref{S4:item} $\Leftrightarrow$ 
\ref{S5:item} $\Leftrightarrow$ 
\ref{S6:item} $\Leftrightarrow$ 
\ref{S7:item} $\Leftrightarrow$ 
\ref{S8:item}. 
The implications \ref{S9:item} $\Rightarrow$ \ref{S10:item} and \ref{S1:item} $\Rightarrow$ \ref{S11:item} are obvious. 

\medskip

\ref{S8:item} $\Rightarrow$ \ref{S9:item}. 
Pick two convergent nets $(p_i)_{i \in I}$ and $(q_i)_{i \in I}$ in $X$ with limits $p$ and $q$, respectively, such that $p_i \succsim q_i$ for all $i \in I$. 
Fix also $y \in U$. 
Then, by hypothesis, 
$\langle p_i,y\rangle \ge \langle q_i,y\rangle$ for all $i \in I$ and, by the continuity of the duality map, $\langle p,y\rangle \ge \langle q,y\rangle$. 
By the arbitrariness of $y$, it follows that the binary relation $\succsim$ is continuous. 
It is routine to check that the preorder $\succsim$ is also satisfies also the independence axiom. This concludes the proof. 
\end{proof}

Theorem \ref{thm:closedconesdecisiontheory} can be seen as the underlying tool in the literature on expected multi-utility theorems. 
For instance, the original characterization by Dubra, Maccheroni, and Ok \cite{Dubra_et_al} considers the dual pair $(X,Y)$ where $X=\mathrm{ca}(Z)$ is the vector space of countably additive finite signed Borel measures on a compact metric space $Z$ and $Y$ is the vector space of continuous functions on $Z$. In fact, once $X$ is equipped with the total variation norm, 
its topological dual is isometrically isomorphic to $Y$, cf. also \cite[Appendix A.1]{MR3957335}. 
(A similar remark applies to Evren's extension for preferences on Borel probability measures with compact support over a $\sigma$-compact metric space, see \cite{MR2398816} for details.)

Another important example is the recent article of Ok and Weaver \cite{OkWeaver}, where the authors prove an expected multi-utility theorem where the set $U$ is made by Lipschitz continuous utility functions. Here, assuming that $Z$ is a separable metric space, they use the dual pair $(X,Y)$, where $X$ stands for the Kantorovich–Rubinstein space $\mathrm{KR}(Z)$ and $Y=\mathrm{Lip}_0(Z)$ for the set of all Lipschitz continuous functions on $Z$ that vanish at a given point $z_0 \in Z$, see \cite[Theorem 2.1 and Theorem 4.1]{OkWeaver}. Also, the uniqueness part of their characterization, which is \cite[Theorem 5.1]{OkWeaver}, is recovered by the equivalence \ref{S6:item} $\Longleftrightarrow$ \ref{S8:item} in Theorem \ref{thm:closedconesdecisiontheory}. A similar comment applies to the Anscombe--Aumann analogue \cite[Theorem 4.2]{OkWeaver}, cf. also \cite[Proposition 4]{Macch} (we omit details).

Considering that, at the end of the story, one would like to obtain an expected multi-utility representation as in \ref{S8:item}, the \textquotedblleft technical difficulty\textquotedblright\, reduces to prove condition \ref{S1:item}, that is, the set $C$ defined in \eqref{eq:definitionC} is a weak-closed convex cone in $X$. In fact, this has been the strategy of the proofs in the above examples, see \cite[Claim 1]{Dubra_et_al} and \cite[Lemma 4.2]{OkWeaver}, respectively.

It is also worth to remark that the missing implications in the statement of Theorem \ref{thm:closedconesdecisiontheory} do not hold in general. 
E.g., the implication \ref{S11:item} $\implies$ \ref{S1:item} does not hold even if $C$ is a vector subspace of $X$. For, it is known that $\ell_1$ is a dense sequentially weak-closed subspace of the topological dual $\ell_\infty^\prime$. As a related result, considering the dual pair $(X,X^\prime)$, it is also known \emph{every} real normed vector space $X$ admits a sequentially weak-closed subset which is not weak-closed, cf. \cite[Proposition 3.2]{MR3511120}. 
On the other hand, it is remarkable that the Krein--\u{S}mulian theorem states that every sequentially weak$^\star$-closed convex set in the dual of a separable normed space is weak$^\star$-closed, cf. 
\cite[Corollary 2.7.13]{MR1650235}. Related results in dual Banach spaces can be found in \cite{MR3944288, MR3754364, MR3349508}, see also references therein.

\section{Cones of Lotteries}\label{sec:coneslotteries}

In this section, we are going to provide, in the special case of binary relations $\succsim$ on the set of lotteries $\Delta$, several sufficient conditions for the converse missing implications in Theorem \ref{thm:closedconesdecisiontheory}. 

For, fix hereafter an outcome set $Z$ with $|Z|\ge 2$, and endow $\mathbf{R}_0^Z$ 
and all its subsets 
with the order structure induced by the product pointwise order $\le$, so that 
$$
x\le y
\quad \Longleftrightarrow \quad
x(z)\le y(z) \,\,\text{ for all }z \in Z,
$$
for all $x,y \in \mathbf{R}_0^Z$. Hence, the pair $(\mathbf{R}_0^Z, \le)$ is a Dedekind complete Riesz space, see e.g. \cite[p. 129]{MR2317344}. 
Note that the space $\mathbf{R}_0^Z$ can also be endowed with a norm $\|\cdot\|$ defined by 
\begin{equation}\label{eq:normdef}
\|x\|:=\sum\nolimits_{z \in Z}|x(z)|
\end{equation}
for all $x \in \mathbf{R}_0^Z$. 
Accordingly, a lottery $p$ is simply a norm-one vector in the positive cone 
of $\mathbf{R}_0^Z$, i.e., 
$$
\Delta=\{x \in \mathbf{R}_0^Z: \|x\|=1 \text{ and }x(z)\ge 0 \text{ for all }z \in Z\}.
$$
For each $u: Z \to \mathbf{R}$ and $S\subseteq Z$, we denote the support of $u$ by $\mathrm{supp}(u):=\{z \in Z: u(z)\neq 0\}$, and we write $u\upharpoonright S$ for the function $Z\to \mathbf{R}$ defined by
\begin{displaymath}
\forall z \in Z, \quad 
(u\upharpoonright S)(z)=
\begin{cases}
\, u(z) \,\,& \text{ if }z \in S,\\
\, 0 & \text{ otherwise}. 
\end{cases}
\end{displaymath}
Two lotteries $p,q \in \Delta$ are called \emph{orthogonal} if they have disjoint support, which can be rewritten as $|p| \wedge |q|=0$. 
Define the vector space $\Sigma$ of all $x \in \mathbf{R}_0^Z$ for which $\sum\nolimits_z x(z)=0$, or, equivalently,
$$
\Sigma:=\mathrm{span}(\Delta-\Delta).
$$
Lastly, recall that $e \in \mathbf{R}^Z$ stands for the constant function $1$ and, for each $z \in Z$, denote by $e_z$ the unique lottery supported on $z$. 

\begin{rmk}\label{rmk:dualpairconstants}
Let $Z$ be a nonempty set and $(\mathbf{R}_0^Z, W)$ be a dual pair where $W\subseteq \mathbf{R}^Z$ is a vector space containing $e$. Let also $\Theta$ be the vector subspace of $W$ containing all constant utility functions, that is, 
$
\Theta:=\{\theta e: \theta \in\mathbf{R}\}.
$
Accordingly, it is routine to check that 
$$
\left(\Sigma, \,\,\frac{W}{\Theta}\right)
$$
is another dual pair. 
(To this aim, note that $\mathbf{R}_0^Z$ and $W$ are isomorphic to the direct sums $\Sigma \oplus \mathbf{R}$ and $(W/\Theta) \oplus \Theta$, respectively; in addition, $\Theta$ is the annihilator of $\Sigma$, hence the topological dual of $\Sigma$ is isomorphic to $W/\Theta$.)

It is also worth noting that, for each nonempty $U,V\subseteq W/\Theta$, the condition $\overline{\mathrm{co}}(\mathrm{cone}(U))=\overline{\mathrm{co}}(\mathrm{cone}(V))$ in $W/\Theta$ reduces to $\langle \, U\,\rangle=\langle \, V\,\rangle$ in $W$.
\end{rmk}

In the next results, we will need the following definition: 
\begin{defi}\label{def:Lproperty}
A vector space $W\subseteq \mathbf{R}^Z$ has the (\textsc{L})\emph{-property} if $e \in W$ and 
\begin{equation}\label{eq:Lpropertydefinition}
\forall S\subseteq Z, \quad 
u \in W 
\,\,\implies \,\,
u\upharpoonright S \in W.
\end{equation}
\end{defi}

Examples of vector spaces with the (\textsc{L})-property include:
\begin{list}{$\diamond$}{}
\item the whole space $\mathbf{R}^Z$;
\item the space $\ell_\infty^Z$ of bounded functions;
\item the space $B_0(Z)$ of finitely-valued functions;
\item the space $B_\omega(Z)$ of countably-valued functions;
\item the space $B_{0,\omega}(Z)$ of functions $u \in \mathbf{R}^Z$ for which there exists a countable $S\subseteq Z$ such that $u \upharpoonright S \in B_0(Z)$; 
\item given a free ultrafilter $\mathscr{F}$ on $Z$, the space $\ell_\infty(\mathscr{F})$ of all $u \in \mathbf{R}^Z$ such that 
$\{z \in Z: |u(z)|\le k\}\in \mathscr{F}$ for some $k \in \mathbf{R}$; 
\item countable direct sums of vector spaces with the (\textsc{L})-property, provided that $Z$ is infinite.
\end{list}
Note that, if $W$ has the (\textsc{L})\emph{-property}, then $e_z \in W$ for all $z \in Z$. Hence $W$ separates the points of $\mathbf{R}_0^Z$, so that 
$$
(\mathbf{R}_0^Z,W)
$$
is a dual pair, with duality map $\langle x,u\rangle:=\mathbf{E}_x[u]$. 

It is worth to remark that the condition \eqref{eq:Lpropertydefinition} 
defining the 
(\textsc{L})-property has an intuitive interpretation as a rationality type criterion. Indeed, the contuinuity of preferences in the dual pair $(\mathbf{R}_0^Z, W)$, see Definition \ref{defi:continuitypreferences}, imposes that the decision maker considers closeness of lotteries $(p_i)_{i \in I}$ to a lottery $p$ as indistinguishably according to the test functions in $W$: the property that $W$ is closed under multiplication with all indicator functions can be interpreted as the decision maker's ability to realize the same closeness for \emph{all} conditional lotteries. A bit more precisely, if $\lim_i\mathbf{E}_{p_i}[u] \to \mathbf{E}_p[u]$ and $S$ is a nonempty subset of $\mathrm{supp}(p)$, then $\lim_i\mathbf{E}_{p_i}[u\upharpoonright S] \to \mathbf{E}_p[u\upharpoonright S]$ and each term in the latter limit can be rewritten (at least, definitively and up to constants) as $\mathbf{E}_q[u]$ for some suitable conditional lottery $q$.

With these premises, we start with two easy lemmas: 
\begin{lem}\label{lem:decomposition}
For each $x \in \Sigma$, there exist $\alpha\ge 0$ and orthogonal lotteries $p,q \in \Delta$ such that 
$$
x=\alpha(p-q).
$$
In addition, such decomposition is unique whenever $x\neq 0$.
\end{lem}
\begin{proof}
If $x=0$, then $x=0(\delta_{z_1}-\delta_{z_2})$, where $z_1,z_2 \in Z$ are two distinct outcomes (note that it is possible since $|Z|\ge 2$). Hence, suppose hereafter that $x\neq 0$. Set $\alpha:=\sum\nolimits_{z \in Z}x^+(z)$, which is positive since $x$ is a nonzero element of $\Sigma$. The claim follows by the uniqueness of the decomposition $x=x^+-x^-$ in orthogonal vectors, see e.g. \cite[Chapter 1]{MR2011364}, and by setting $p:=x^+/\alpha$ and $q:=x^-/\alpha$. 
\end{proof}

Hereafter, if $\succsim$ is a binary relation over the set of lotteries $\Delta$, then it is also regarded as a binary relation over $\Sigma$ (or, more  generally, over any superset of $\Delta$), with the convention that $x\succsim y$ does not hold if $x \notin \Delta$ or $y \notin \Delta$.
\begin{lem}\label{lem:disjointsupports}
Let $\succsim$ be a reflexive binary relation on $\Delta$ which satisfies the independence axiom, and define the cone $C$ as in \eqref{eq:definitionC} with $X=\Sigma$. Then, for each nonzero $x \in C$, there exist a unique $\alpha>0$ and unique orthogonal lotteries $p,q \in \Delta$ such that 
\begin{equation}\label{eq:claimdecompositionlotteries}
x=\alpha (p-q) \quad \text{ and }\quad p\succsim q.
\end{equation}
\end{lem}
\begin{proof}
If $C=\{0\}$ the claim holds. 
Otherwise, fix a nonzero vector $x \in C$. 
By the definition of $X$ and since $\succsim$ can be regarded as a subset of $\Delta^2$, there exist a real $\alpha>0$ and lotteries $p,q \in \Delta$ which satisfy \eqref{eq:claimdecompositionlotteries}. 
Note that, however, $p$ and $q$ are not necessarily orthogonal. 

Hence, assume hereafter that $p$ and $q$ are not orthogonal, so that $p\neq q$ and $p\neq (p-q)^+$. 
The positive part $(p-q)^+:=(p-q) \vee 0$ and the negative part $(p-q)^-:=(q-p) \vee 0$ are both nonzero, which implies that $\kappa:=\sum\nolimits_z (p-q)^+(z)$ is a real in the open interval $(0,1)$. 
At this point, define the lotteries 
$$
a:=\frac{(p-q)^+}{\kappa}
\quad 	\text{ and }\quad 
b:=\frac{(p-q)^-}{\kappa}.
$$
Note that 
$
p
=(p-q)^++(p \wedge q)=\kappa a+(1-\kappa)r,
$ 
where $r$ is the well-defined lottery $(p\wedge q)/\kappa \in \Delta$. With an analogous identity for $q$, we obtain 
$$
\kappa a+(1-\kappa)r=p\succsim q=\kappa b+(1-\kappa)r.
$$
Since $\succsim$ satisfies the independence axiom, we obtain that $a \succsim b$ and, by construction, $a$ and $b$ are orthogonal and $x=\lambda (a-b)$, with $\lambda:=\alpha \kappa$. The uniqueness of the representation follows by Lemma \ref{lem:decomposition}.
\end{proof}

The converse implications that we are going to prove are drawn in the dashed arrows of Figure \ref{fig:converseimplications} below. 

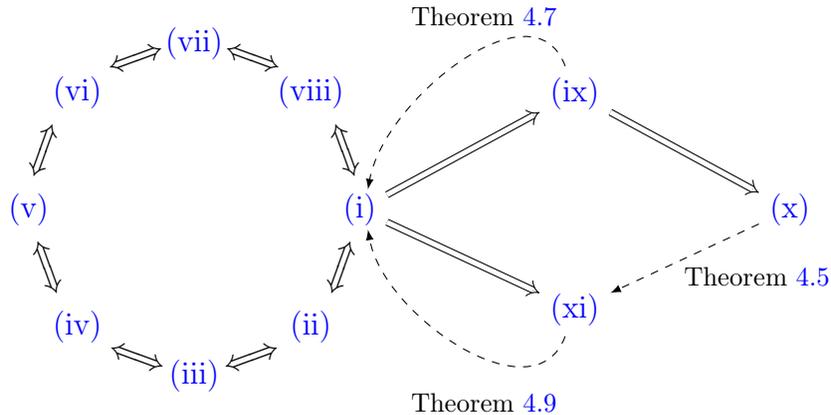
\begin{figure}[!htbp]
\centering
\begin{tikzpicture}
[scale=2.2]
\node (a1) at (0,1){\ref{S7:item}};
\node (a2) at (0.705,0.705){\ref{S8:item}};
\node (a3) at (1,0){\ref{S1:item}};
\node (a4) at (0.705,-0.705){\ref{S2:item}};
\node (a5) at (0,-1){\ref{S3:item}};
\node (a6) at (-0.705,-0.705){\ref{S4:item}};
\node (a7) at (-1,0){\ref{S5:item}};
\node (a8) at (-0.705,0.705){\ref{S6:item}};

\node (b12) at (.35,1-.1){\rotatebox[origin=c]{-20}{$\Longleftrightarrow$}};
\node (b18) at (-.35,1-.1){\rotatebox[origin=c]{20}{$\Longleftrightarrow$}};
\node (b56) at (-.35,-1+.1){\rotatebox[origin=c]{-20}{$\Longleftrightarrow$}};
\node (b54) at (.35,-1+.1){\rotatebox[origin=c]{20}{$\Longleftrightarrow$}};
\node (b23) at (1-.1, .35){\rotatebox[origin=c]{-70}{$\Longleftrightarrow$}};
\node (b34) at (1-.1, -.35){\rotatebox[origin=c]{70}{$\Longleftrightarrow$}};
\node (b67) at (-1+.1, .35){\rotatebox[origin=c]{70}{$\Longleftrightarrow$}};
\node (b78) at (-1+.1, -.35){\rotatebox[origin=c]{-70}{$\Longleftrightarrow$}};

\node (a9) at (2.3,.7){\ref{S9:item}}; 
\node (a10) at (3.6,0){\ref{S10:item}}; 
\node (a11) at (2.3,-.6){\ref{S11:item}}; 

\draw[-implies,double equal sign distance] (a3) -- (a9);
\draw[-implies,double equal sign distance] (a9) -- (a10);
\draw[-implies,double equal sign distance] (a3) -- (a11);

\draw[-latex, dashed, thin] (a10) to (a11);
\draw[-latex, dashed, thin] (a11) to [out=-110,in=-80] (1.05,-.12); 
\draw[-latex, dashed, thin] (a9) to [out=110,in=80] (1.05,.12); 

\node (a10a11) at (3.4,-.4){{\footnotesize Theorem \ref{thm:a10a11}}};
\node (a11sotto) at (1.75,-1.16){{\footnotesize Theorem \ref{thm:a11a1}}};
\node (a9sopra) at (1.75,1.16){{\footnotesize Theorem \ref{thm:a9a1}}};
\end{tikzpicture}
\caption{Converse implications for binary relations on $\Delta$.  \label{fig:converseimplications}}
\end{figure}

\begin{thm}\label{thm:a10a11}
Let $W$ be a vector space with the \textup{(}\textsc{L}\textup{)}-property and consider the dual pair $(\mathbf{R}_0^Z, W)$. 
Also, let $\succsim$ be a reflexive binary relation on $\Delta$ satisfying the independence axiom, and define the cone $C$ as in \eqref{eq:definitionC} with $X=\mathbf{R}_0^Z$. 
Then 
$$
\ref{S10:item} \implies \ref{S11:item}.
$$
\end{thm}
\begin{proof}
Let $(x_n)_{n\ge 1}$ be a sequence in $C$ which is weak-convergent to some $x_0 \in \mathbf{R}_0^Z$. Since $e \in W$ then $\Sigma$ is weak-closed, so $x_0 \in \Sigma$. To complete the proof, we need to show that $x_0 \in C$. If $x_0=0$ the claim holds, otherwise suppose hereafter that $x_0\neq 0$ and, without loss of generality, $x_n\neq 0$ for all $n \ge 1$ (since the weak-topology is Hausdorff). 

It follows by Lemma \ref{lem:disjointsupports} that there exist unique positive reals $(\alpha_n)_{n\ge 0}$ and unique pairs of orthogonal lotteries $((p_n,q_n))_{n\ge 0}$ such that 
$$
\forall n\ge 1, \quad \,\,\,
p_n \succsim q_n \,\,\text{ and }\,\, x_n=\alpha_n(p_n-q_n).
$$
and 
$$
x_0=\alpha_0(p_0-q_0).
$$
Define $P_n:=\mathrm{supp}(p_n)$, $Q_n:=\mathrm{supp}(Q_n)$ and $S_n:=P_n\cup Q_n$ for all $n\ge 0$. 
By construction, we have $P_n\cap Q_n=\emptyset$. 
Note that $\langle x_n, e_z\rangle=x_n(z)$ for each $z \in Z$, hence the hypothesis that $(x_n)_{n\ge 1}$ is weak-convergent to $x_0$ implies 
\begin{equation}\label{eq:pointwiselimit}
\forall z \in Z, \quad 
\lim_{n\to \infty}x_n(z)=x_0(z). 
\end{equation}

\begin{claim}\label{claim1}
$\lim_{n}\alpha_np_n(z)=\alpha_0p_0(z)$ for all $z \in P_0$. 
\end{claim}
\begin{proof}
Pick $z \in P_0$. Since $p_0$ and $q_0$ are orthogonal, for each $\varepsilon \in (0,\alpha_0 p_0(z))$ there exists $n_\varepsilon \ge 1$ such that
$$
\forall n\ge n_\varepsilon, \quad 
\left|\alpha_n(p_n-q_n)(z)-\alpha_0 p_0(z)\right|<\varepsilon.
$$
However, the lotteries $p_n$ and $q_n$ are orthogonal, so $z \in P_n$ for all $n\ge n_\varepsilon$. 
By the arbitrariness of $\varepsilon$, we obtain $\lim_{n}\alpha_np_n(z)=\alpha_0p_0(z)$. 
\end{proof}

\begin{claim}\label{claim2}
The sequence $(\alpha_np_n)_{n\ge 1}$ is weak-convergent to $\alpha_0p_0$. 
\end{claim}
\begin{proof}
First of all, since $(x_n)_{n\ge 1}$ is weak-convergent to $x_0$, we have by definition $\lim_n\langle x_n-x_0,u\rangle=0$ for all $u \in W$, or equivalently
\begin{equation}\label{eq:alphapandq}
\forall u\in W, \quad \lim_{n\to \infty}\left(\langle \alpha_np_n-\alpha_0p_0,u\rangle - \langle \alpha_nq_n-\alpha_0q_0,u\rangle \right)=0.
\end{equation}
As it follows from \eqref{eq:pointwiselimit}, we have $\lim_n x_n(z)=0$ for all $z\in Z\setminus S_0$. 
For such values of $z$ we have also $\alpha_0p_0(z)=\alpha_0q_0(z)=0$, so that $\lim_n \alpha_n(p_n(z)-q_n(z))=0$ and, considering that $P_n\cap Q_n=\emptyset$, we have
\begin{equation}\label{eq:dkjhgkhrd}
\forall z \in Z\setminus S_0, \quad 
\lim_{n\to \infty} \,\langle \alpha_np_n, e_z\rangle
=\lim_{n\to \infty} \,\langle \alpha_nq_n, e_z\rangle
=0.
\end{equation}
Let us suppose for the sake of contradiction that 
$$
\limsup_{n\to \infty}\, |\langle \alpha_np_n-\alpha_0p_0,u_0\rangle| >0
$$
for some $u_0 \in W$. 
It follows by Claim \ref{claim1} that we can suppose without loss of generality that $\mathrm{supp}(u_0)\subseteq Z\setminus S_0$. In addition, rescaling $u_0$ if necessary and taking into account \eqref{eq:alphapandq}, there exists a strictly increasing sequence $(n_k)_{k\ge 1}$ of positive integers such that 
\begin{equation}\label{eq:dkjhgkhrd2}
\lim_{k\to \infty}\, \langle \alpha_{n_k}p_{n_k},u_0\rangle
=\lim_{k\to \infty}\, \langle \alpha_{n_k}q_{n_k},u_0\rangle
=1
\end{equation}
or 
\begin{equation}\label{eq:dkjhgkhrd3}
\lim_{k\to \infty}\, \langle \alpha_{n_k}p_{n_k},u_0\rangle
=\lim_{k\to \infty}\, \langle \alpha_{n_k}q_{n_k},u_0\rangle
=\infty.
\end{equation}
Note that \eqref{eq:dkjhgkhrd} implies that $u_0$ is not finitely-supported. 

\medskip

First, assume that \eqref{eq:dkjhgkhrd2} holds. We are going to define recursively a strictly increasing sequence $(k_m)_{m\ge 1}$ of positive integers as it follows: 
Set by convention $k_0:=0$ and $S_{n_{k_0}}:=\emptyset$ and, if all positive integers $k_1,\ldots,k_{m-1}$ have been defined, pick the smallest integer $k_m>k_{m-1}$ such that 
\begin{equation}\label{eq:uglycond1}
\max\left\{
\left|\langle \alpha_{n_{k_m}}p_{n_{k_m}},u_0\rangle-1\right|,
\left|\langle \alpha_{n_{k_m}}q_{n_{k_m}},u_0\rangle-1\right|
\right\}<2^{-m}
\end{equation}
and
$$
\forall n\ge n_{k_m}, \quad 
\max\{|\langle \alpha_np_n, v_m\rangle|,
|\langle \alpha_nq_n, v_m\rangle| \}<2^{-m},
$$
where $v_m \in \mathbf{R}^Z$ is the finitely-supported function defined by $v_m(z):=u_0(z)$ if $z \in S_{n_{k_1}}\cup \cdots \cup S_{n_{k_{m-1}}}$ and $v_m(z):=0$ otherwise. Note that such sequence is well defined, thanks to \eqref{eq:dkjhgkhrd} and \eqref{eq:dkjhgkhrd2}. 
Lastly, define the function $u_1 \in \mathbf{R}^Z$ by 
\begin{displaymath} 
u_1(z):=
\begin{cases}
\,2u_0(z)& \text{if }z \in P_{n_{k_m}}\setminus (S_{n_{k_1}}\cup \cdots \cup S_{n_{k_{m-1}}}) \text{ for some }m\ge 1,\\
\,u_0(z)& \text{otherwise}
\end{cases}
\end{displaymath}
for all $z \in Z$. 
Since $W$ is a vector space with the (\textsc{L})-property, $u_1 \in W$ and $v_m \in W$ for all $m\ge 1$. 
It follows by construction that, for all integers $m\ge 1$,
\begin{displaymath}
\begin{split}
\langle \alpha_{n_{k_m}}p_{n_{k_m}},u_1\rangle &=
\langle \alpha_{n_{k_m}}p_{n_{k_m}},u_0\rangle +\langle \alpha_{n_{k_m}}p_{n_{k_m}},u_0-v_m\rangle \\
&\ge (1-2^{-m})+(1-2^{-m}-2^{-m}) \ge 2-2^{2-m}
\end{split}
\end{displaymath}
and, on the other hand, 
\begin{displaymath}
\langle \alpha_{n_{k_m}}q_{n_{k_m}},u_1\rangle =
\langle \alpha_{n_{k_m}}q_{n_{k_m}},u_0\rangle \le 1+2^{-m}.
\end{displaymath}
Therefore $u_1$ is a function in $W$ with infinite support disjoint from $S_0$ and 
\begin{displaymath}
\begin{split}
\limsup_{n\to \infty}\,\langle x_n,u_1\rangle
&=\limsup_{n\to \infty}\, (\langle \alpha_np_n,u_1\rangle- \langle \alpha_nq_n,u_1\rangle) \\
&\ge \limsup_{m \to \infty}\, (\langle \alpha_{n_{k_m}}p_{n_{k_m}},u_1\rangle- \langle \alpha_{n_{k_m}}q_{n_{k_m}},u_1\rangle) 
\ge 1,
\end{split}
\end{displaymath}
which contradicts \eqref{eq:alphapandq}.

\medskip

In the second case, i.e., assuming \eqref{eq:dkjhgkhrd3}, we can construct an analogous strictly increasing sequence $(k_m)_{m\ge 1}$ of positive integers replacing condition \eqref{eq:uglycond1} with 
$$
\langle \alpha_{n_{k_m}}p_{n_{k_m}},u_0\rangle \ge m 
\quad \text{ and }\quad 
|\langle \alpha_{n_{k_m}}p_{n_{k_m}},u_0\rangle-\langle \alpha_{n_{k_m}}q_{n_{k_m}},u_0\rangle|<2^{-m}
$$
for all $m\ge 1$.
Indeed, proceeding similarly, it would follow that 
\begin{displaymath}
\begin{split}
\langle \alpha_{n_{k_m}}p_{n_{k_m}},u_1\rangle&- \langle \alpha_{n_{k_m}}q_{n_{k_m}},u_1 \rangle
=  
\langle \alpha_{n_{k_m}}p_{n_{k_m}},2u_0-v_m\rangle- \langle \alpha_{n_{k_m}}q_{n_{k_m}},u_0\rangle \\
&\ge  
\langle \alpha_{n_{k_m}}p_{n_{k_m}},2u_0-v_m\rangle- \langle \alpha_{n_{k_m}}p_{n_{k_m}},u_0\rangle-2^{-m}\\
&=
\langle \alpha_{n_{k_m}}p_{n_{k_m}},u_0\rangle- \langle \alpha_{n_{k_m}}p_{n_{k_m}},v_m\rangle -2^{-m}\\
&\ge m-2^{1-m}
\end{split}
\end{displaymath}
for all $m\ge 1$. Since the support of $u_1$ has empty intersection with $S_0$, we reached another contradiction with \eqref{eq:alphapandq}, completing the proof of the claim. 
\end{proof}

\begin{claim}\label{claim3}
$\lim_{n}\alpha_n=\alpha_0$.
\end{claim}
\begin{proof} 
It follows by Claim \ref{claim2} since $e \in W$ and $\lim_n \langle \alpha_np_n,e\rangle=\langle \alpha_0p_0,e\rangle$. 
\end{proof}

\begin{claim}\label{claim4}
The sequence $(p_n)_{n\ge 1}$ is weak-convergent to $p_0$. 
\end{claim}
\begin{proof}
Fix $u \in W$ and a positive integer $n_0$ such that $|\alpha_n|\ge \frac{1}{2}|\alpha_0|$ for all $n\ge n_0$ which exists by Claim \ref{claim3}. Then, for all $n\ge n_0$, we have 
\begin{displaymath}
\begin{split}
|\langle p_n-p_0, u\rangle |
&=\left|\left\langle \frac{\alpha_np_n-\alpha_0p_0}{\alpha_n}+\frac{\alpha_0p_0}{\alpha_n}-\frac{\alpha_0p_0}{p_0}, u\right\rangle \right|\\
&\le 
\frac{1}{|\alpha_n|}\left|\left\langle \alpha_np_n-\alpha_0p_0,u\right\rangle \right|+\left|\frac{1}{\alpha_n}-\frac{1}{\alpha_0}\right|\cdot |\langle \alpha_0p_0,u\rangle|\\
&\le 
\frac{2}{|\alpha_0|}\left|\left\langle \alpha_np_n-\alpha_0p_0,u\right\rangle \right|+\frac{2|\alpha_n-\alpha_0|}{|\alpha_0|^2}|\langle \alpha_0p_0,u\rangle|.\\
\end{split}
\end{displaymath}
Taking the limit $n\to \infty$, it follows that $\lim_n \langle p_n,u\rangle=\langle p,u\rangle$.
\end{proof}

Reasoning analogously as in Claim \ref{claim4}, we obtain that the sequence $(q_n)_{n\ge 1}$ is weak-convergent to $q_0$. 
Together with the sequential continuity of $\succsim$, it follows that $p_0\succsim q_0$. Therefore $x_0=\alpha_0(p_0-q_0)$ belongs to $C$, concluding the proof. 
\end{proof}

At this point, we can state the following necessary and sufficient conditions for the existence of an expected multi-utility representation on the set of lotteries:
\begin{cor}\label{cor:necesuff}
Let $Z$ be a nonempty set and consider the dual pair $(\mathbf{R}_0^Z,W)$, where $W\subseteq \mathbf{R}^Z$ has the \textup{(}\textsc{L}\textup{)}-property. 
Then the following are equivalent\emph{:}
\begin{enumerate}[label={\rm (\alph{*})}]
\item \label{item:mainA} every sequentially weak-closed convex cone in $\Sigma$ is weak-closed\emph{;}
\item \label{item:mainB} for each sequentially continuous preorder $\succsim$ on the set of lotteries $\Delta$ which satisfies the independence axiom, there exists a nonempty set $U\subseteq W$ of utility functions $u:Z\to \mathbf{R}$ such that, for all lotteries $p,q \in \Delta$, 
$$
p\succsim q 
\,\,\,\Longleftrightarrow\,\,\,
\mathbf{E}_p[u] \ge \mathbf{E}_q[u]
\,\text{ for all }u \in U\emph{;}
$$
\item \label{item:mainC} condition \ref{item:mainB} holds and, in addition, if $V\subseteq W$ is another nonempty set of utility functions with the same property then 
$\langle\,U\,\rangle=\langle\,V\,\rangle$.
\end{enumerate}
\end{cor}
\begin{proof}
\ref{item:mainA} $\implies$ \ref{item:mainB} $\Longleftrightarrow$ \ref{item:mainC}. Fix a sequentially continuous preorder $\succsim$ on $\Delta$ which satisfies the independence axiom, i.e., assume \ref{S10:item}. Thanks to 
Theorem \ref{thm:a10a11}, 
the cone $C$ defined in  \eqref{eq:definitionC} is a sequentially weak-closed convex cone in $\Sigma$. Hence, by the standing hypothesis \ref{item:mainA}, $C$ is necessarily weak-closed, i.e., \ref{S1:item} holds. The claim follows by equivalences \ref{S1:item} $\Longleftrightarrow$ \ref{S6:item} $\Longleftrightarrow$ \ref{S8:item} in Theorem \ref{thm:closedconesdecisiontheory} applied to the dual pair $(\Sigma, W/\Theta)$, cf. Remark \ref{rmk:dualpairconstants}.

\medskip

\ref{item:mainB} $\implies$ \ref{item:mainA}. 
Fix a sequentially weak-closed convex cone $C\subseteq \Sigma$ and let 
$\succsim$ be the binary relation on $\Delta$ defined by $p \succsim q$ if and only if $p-q \in C$. Note that $\succsim$ satisfies the independence axioms (we omit details). Since $C$ is convex, then $\succsim$ is transitive, cf. \cite[Lemma A.2(d)]{MR3957335}. In addition, if $(p_n)_{n\ge 1}$ and $(q_n)_{n\ge 1}$ are sequences of lotteries which are weak-convergent to $p$ and $q$ in $\Delta$, respectively, and $p_n\succsim q_n$ for all $n\ge 1$, then necessarily $(p_n-q_n)_{n\ge 1}$ is weak-convergent by $p-q$. Since $p_n\succsim q_n$ is equivalent to $p_n-q_n \in C$, it follows by the sequential weak-closure of $C$ that $p-q \in C$, so that $p\succsim q$. To sum up, $\succsim$ is a sequentially continuous preorder satisfying the independence axiom. 
By condition \ref{item:mainB}, $\succsim$ satisfies \ref{S8:item}. 
Hence, the equivalence \ref{S1:item} $\Longleftrightarrow$ \ref{S8:item} given in Theorem \ref{thm:closedconesdecisiontheory} implies that $C$ has to be weak-closed. 
\end{proof}

Note that, if the binary relation $\succsim$ is assumed to be continuous, as in \ref{S9:item} (and not only sequentially continuous, as in \ref{S10:item}), then every passage in the proof of Theorem~\ref{thm:a10a11} still works for nets in place of sequences, with the unique exception of Claim~\ref{claim2}. 
Indeed, it is not true, in general, that every net $(x_i)_{i \in I}$ admits a subsequence (i.e., a countable subnet). 
Hence, the construction of a subsequence $(x_{n_{k_m}})_{m\ge 1}$ such that $\limsup_m |\langle x_{n_{k_m}},u_0\rangle|> 0$ for some $u_0$ with support disjoint from $S_0$ does not necessarily imply a contradiction with the hypothesis $\lim_{i}\langle x_i,u_0\rangle =0$. 
This observation suggests the following result: 

\begin{thm}\label{thm:a9a1}
Let $Z$ be a nonempty set and $(\mathbf{R}_0^Z, W)$ be a dual pair where $W\subseteq \mathbf{R}^Z$ is a vector space such that $e \in W$ and the map $T: \mathbf{R}_0^Z\to \mathbf{R}_0^Z$ defined by $x\mapsto x^+$ 
is weak-to-weak continuous\emph{.} 

Also, let $\succsim$ be a reflexive binary relation on $\Delta$ which satisfies the independence axiom, and define the cone $C$ as in \eqref{eq:definitionC} with $X=\mathbf{R}_0^Z$. 
Then 
$$
\ref{S9:item} \implies \ref{S1:item}.
$$
\end{thm}
\begin{proof} 
Note that $\Sigma$ is weak-closed since $e \in W$. 
The proof of Theorem \ref{thm:a9a1} goes verbatim as the proof of Theorem \ref{thm:a10a11}, replacing sequences $(x_n)_{n\ge 1}$ with nets $(x_i)_{i \in I}$ 
and noting, in place of Claim \ref{claim1} and Claim \ref{claim2}, that the 
weak-to-weak 
continuity of $T$ implies $\lim_i \alpha_ip_i=\lim_i Tx_i=T(\lim_i x_i)=\alpha_0p_0$.
\end{proof}

\begin{rmk}\label{rmk:consistentlocallysolid}
A sufficient condition to imply the continuity of $T$ is that 
there exists a consistent topology $\tau$ on $\mathbf{R}_0^Z$ which is 
locally solid, namely, it has local base $(U_\alpha)_{\alpha \in A}$ of neighborhoods at $0$ such that 
$$
\forall \alpha \in A, \forall p \in U_\alpha, \forall q \in \mathbf{R}_0^Z, \quad 
|q|\le |p| \implies q \in U_\alpha.
$$
Indeed, it follows by a characterization of Roberts and Namioka \cite[Theorem 2.17]{MR2011364} that $T$ is uniformly continuous. In particular, $T$ is $\tau$-to-$\tau$ continuous and, hence, also weak-to-weak continuous, cf. \cite[Theorem 6.17]{MR2378491}.

However, as it has been shown in \cite[Example 2.18]{MR2011364} there exists a topological Riesz space such that the map $x\mapsto x^+$ is continuous and, on the other hand, its topology is not locally solid. 
\end{rmk}

\begin{thm}\label{thm:a11a1}
Let $Z$ be a nonempty countable set.
$$
\ref{S11:item} \implies \ref{S1:item}.
$$
\end{thm}
\begin{proof}
Let $C\subseteq \Sigma$ be a sequentially weak-closed convex cone. Then $C$ is weak-closed if $Z$ is finite. Otherwise, if $Z$ is countably infinite, $C\cap F$ is relatively closed in $F$ whenever $F$ is a finite dimensional subspace of $\Sigma$. 
It follows by \cite[Proposition 8.5.28]{MR880207} that $C$ is weak-closed, cf. also \cite[Corollary 2.1]{MR2070962}. 
\end{proof}

We conclude this section with an application which relies on the stronger hypothesis of \emph{continuous} preorders satisfying the independence axiom, in place of sequential continuity.  
For, given a nonempty set $Z$, let $\zeta$ be the topology on $\mathbf{R}_0^Z$ so that a local base of open neighboorhoods at $0$ is given by the family $\{U_h: h \in (0,\infty)^Z\}$, where 
$$
U_h:=\{x \in \mathbf{R}_0^Z: |x(z)|<h(z) \text{ for all }z \in Z\}.
$$
The topology $\zeta$ has been used, in the case $Z$ is countably infinite, by Kannai in \cite{MR170713} to solve an open question of Aumann \cite{MR174381}. As we are going to show in the proof below, the topological dual of $(\mathbf{R}_0^Z, \zeta)$ is precisely the vector space $c_\omega(Z)$ of countably-supported utility functions. Analogously, we endow $c_\omega(Z)$ with the weak-topology $\sigma(c_\omega(Z), \mathbf{R}_0^Z)$. 
\begin{thm}\label{thm:mainintermediate}
Let $Z$ be a nonempty set and $\succsim$ be a binary relation over $\Delta$. 
Then $\succsim$ is $\zeta$-continuous preorder which satisfies the independence axiom if and only if there exists a nonempty set $U\subseteq c_{\omega}(Z)$ of countably-supported utility functions such that \eqref{eq:claimedmultiutility} holds for all $p,q \in \Delta$. 

In addition, if $V\subseteq c_{\omega}(Z)$ is another nonempty set of countably-supported utility functions with the same property, then $\langle\, U\,\rangle=\langle\,V\,\rangle$.
\end{thm}

Note that, in the special case where $Z$ is countable, Theorem \ref{thm:mainintermediate} gives another proof of Theorem \ref{thm:mainweaktopology}, although under the stronger hypothesis of $\zeta$-continuity of $\succsim$ (indeed, in such case, the topologies $\zeta$ and $w$ are consistent). At the same time, Theorem \ref{thm:mainintermediate} holds also if $Z$ is uncountable. 

\begin{proof}
[Proof of Theorem \ref{thm:mainintermediate}]
As anticipated, we start proving that the topological dual of $(\mathbf{R}_0^Z, \zeta)$ is precisely $c_\omega(Z)$:
\begin{claim}
A utility function $u \in\mathbf{R}^Z$ induces a $\zeta$-continuous linear functional $\mathbf{R}_0^Z\to \mathbf{R}$ through $x\mapsto \langle x,u\rangle$ if and only if $u$ is countably-supported. 
\end{claim}
\begin{proof}
\textsc{If part.} Suppose that the support of $u$ is finite or countably infinite. In the first case, it is clear that the map $x\mapsto \langle x,u\rangle$ is $\zeta$-continuous. In the latter, assume without loss of generality that $Z=\mathbf{N}$, fix $\varepsilon>0$, and define 
$$
\forall z \in \mathbf{N}, \quad 
h(z):=\frac{\varepsilon}{1+2^z|u(z)|}.
$$
It follows that 
$$
\forall x \in U_h, \quad 
|\langle x,u\rangle| \le \sum_{z=1}^\infty h(z)u(z)<\varepsilon,
$$
proving the claim. 

\medskip

\textsc{Only If part.} Suppose for the sake of contradiction that the support of $u$ is uncountable. 
Then there exists a rational $q_0>0$ such that $Z_0:=\{z \in Z: |u(z)|\ge q_0\}$ is uncountable. 
Since $u$ is $\zeta$-continuous, there exists $h \in (0,\infty)^Z$ such that $|\langle x,u\rangle| <1$ for all $x \in U_h$. 
With the same reasoning, there exists a rational $q_1>0$ such that $Z_1:=\{z \in Z_0: h(z)\ge q_1\}$ is uncountable. 
Hence there exists a countably infinite subset $\{z_n: n\ge 1\}\subseteq Z_1$. 
At this point, for each $n\ge 1$, define the sequence 
$$
x_n:=q_1\sum_{i=1}^n \mathrm{sgn}(u(z_i))\, e_{z_i},
$$
where $\mathrm{sgn}(t):=1$ if $t>0$, $\mathrm{sgn}(t):=-1$ if $t<0$, and $\mathrm{sgn}(0):=0$. It follows that, for each $n\ge 1$, $x_n \in U_h$ and, on the other hand, 
$$
\langle x_n,u\rangle=\sum_{z \in Z} x_n(z)u(z) 
=q_1\sum_{i=1}^n\sum_{z \in Z} |u(z)| e_{z_i}(z) \ge nq_0q_1.
$$
Hence, we reach a contradiction if $n$ is sufficiently large.
\end{proof}

At this point, consider the dual pair $(\mathbf{R}_0^Z, c_\omega(Z))$ and observe that $\zeta$ is consistent with the weak-topology $\sigma(\mathbf{R}_0^Z, c_\omega(Z))$. Since $\zeta$ is clearly locally solid, the \textsc{Only If} part follows putting together Theorem \ref{thm:closedconesdecisiontheory}, Theorem \ref{thm:a9a1}, together with Remark \ref{rmk:dualpairconstants} and Remark \ref{rmk:consistentlocallysolid}. 
Conversely, the \textsc{If} part is straighforward. 
\end{proof}


\section{Proof of Main Results}\label{sec:mainproofs}

\begin{proof}
[Proof of Theorem \ref{thm:mainweaktopology} and Theorem \ref{thm:mainweaktopologybounded}]
The \textsc{If} part is straightforward, and the \textsc{Only If} part follow by Corollary \ref{cor:necesuff} and Theorem \ref{thm:a11a1}, setting $W=\mathbf{R}^{Z}$ and $W=\ell_\infty^Z$, respectively. Lastly, the uniqueness part is consequence of the equivalence \ref{S6:item} $\Longleftrightarrow$ \ref{S8:item} in Theorem \ref{thm:closedconesdecisiontheory} applied on the dual pair $(\Sigma, W/\Theta)$, cf. Remark \ref{rmk:dualpairconstants}.
\end{proof}

\medskip

We proceed with a related result on monotone binary relations. 
For, let $Z$ be a nonempty countable set, and endow it with a binary relation $\trianglerighteq$. Also, let $W$ be a vector space with the (\textsc{L})-property, and consider the dual pair $(\mathbf{R}_0^Z, W)$. Hence:
\begin{enumerate}
\item [(i)] a binary relation $\succsim$ on $\Delta$ is $\trianglerighteq$-\emph{monotone} if, for all outcomes $a,b \in Z$, $a\trianglerighteq b$ implies $e_a\succsim e_b$. 
\item [(ii)] a utility function $u$ is $\trianglerighteq$-\emph{increasing} if, for all outcomes $a,b \in Z$, $a\trianglerighteq b$ implies $u(a) \ge u(b)$. 
\end{enumerate}
Then we have the following variant of our main results:
\begin{cor}\label{rmk:monotonicity}
Let $Z$ be a nonempty countable set and consider the dual pair $(\mathbf{R}_0^Z,W)$, where $W\subseteq \mathbf{R}^Z$ has the \textup{(}\textsc{L}\textup{)}-property. 
A binary relation $\succsim$ on $\Delta$ is a $\trianglerighteq$-monotone sequentially continuous preorder which satisfies the independence axiom if and only if there exists a nonempty set $U\subseteq W$ of $\trianglerighteq$-increasing utility functions $u: Z\to \mathbf{R}$ such that \eqref{eq:claimedmultiutility} holds for all $p,q \in \Delta$. 

In addition, if $V\subseteq W$ is another nonempty set of $\trianglerighteq$-increasing utility functions with the same property, then 
$\langle\,U\,\rangle=\langle\,V\,\rangle$.
\end{cor}
\begin{proof}
The proof goes on the same lines of the one of Theorem \ref{thm:mainweaktopology} and Theorem \ref{thm:mainweaktopologybounded} above.  
The only difference is that, in the \textsc{Only If} part, we need to show that each utility function $u \in U$ is $\trianglerighteq$-increasing. For, set
$$
H:=\{e_a-e_b: a,b \in Z \,\text{ and }\, a\trianglerighteq b\}.
$$
The hypothesis of $\trianglerighteq$-monotonicity of the binary relation $\succsim$ is equivalent to $H\subseteq C$, where $C$ is the cone defined in \eqref{eq:definitionC}. 
Since $C=U^\prime$,  
it follows by Theorem \ref{thm:closedconesdecisiontheory} that 
$$
U\subseteq \overline{\mathrm{co}}(\mathrm{cone}(U))=(\overline{\mathrm{co}}(\mathrm{cone}(U)))^{\prime\prime}=C^\prime \subseteq H^\prime. 
$$
To conclude, pick $y \in U$. Then $y \in H^\prime$, i.e.,  for each $e_a-e_b \in H$, we have $\langle e_a-e_b,y\rangle \ge 0$. Equivalently, 
for all $a,b \in Z$ with $a \trianglerighteq b$, we have $y(a) \ge y(b)$. Therefore $y$ is $\trianglerighteq$-increasing. 
\end{proof}

We continue with the proof of Theorem \ref{thm:mainnegativecharacterization}. 
As a side note, note that, if $\precsim$ provides a counterexample, then it cannot be an equivalence relation. This relies on the simple fact that every subspace of a vector space $X$ is $\sigma(X,X^\star)$-closed. Taking into account Corollary \ref{cor:necesuff} for $W=\mathbf{R}^Z$, the question seems to be related to the family of topological vector spaces $X$ such that every convex sequentially open subsets is open, which has been studied by Snipes in \cite{MR330993}. 
\begin{proof}
[Proof of Theorem \ref{thm:mainnegativecharacterization}]
Thanks to Corollary \ref{cor:necesuff}, it is sufficient to show that there exists a sequentially weak-closed convex cone $C\subseteq \Sigma$ which is not weak-closed. 
To this aim, let $\{Z_1,Z_2\}$ be a partition of $Z$ such that $|Z_1|=|Z_2|$. 
Let also $h: Z_1\to Z_2$ be a bijection and fix $a \in Z_1$. 
For each nonempty finite subset $B\subseteq Z_1$, define 
$$
\tilde{e}(B):=\frac{1}{|B|^2}\sum_{b \in B}(e_b-e_{h(b)}).
$$
Note that $\tilde{e}(B)$ belongs to $\Sigma$ for each finite nonempty subset of $Z_1$. Now, let $\mathscr{B}$ be the family of nonempty finite subsets of $Z_1\setminus \{a\}$ and $C$ be the sequential weak-closure of the convex cone $C_0$, where 
$$
C_0:=\mathrm{co}\left(\mathrm{cone}\left(\left\{\tilde{e}(\{a\})+\tilde{e}(B): B \in \mathscr{B}\right\}\right)\right).
$$
It is routine to check that $C$ is, indeed, a convex cone. To complete the proof, it will be enough to show that $\tilde{e}(\{a\}) \notin C$ and, on the other hand, $\tilde{e}(\{a\})$ belongs to the weak-closure of $C_0$, so that $C$ is a sequentially weak-closed convex cone in $\Sigma$ which is not weak-closed. 

\medskip

First, let us show that $\tilde{e}(\{a\}) \notin C$. For the sake of contradiction, suppose that, for each $n\ge 1$, there exist a positive integer $k_n$, positive scalars $\lambda_{n,1},\ldots,\lambda_{n,k_n} \in \mathbf{R}$, and 
nonempty finite pairwise disjoint
sets $B_{n,1},\ldots,B_{n,k_n}\in \mathscr{B}$ such that the sequence $(p_n)_{n\ge 1}$ is weak-convergent to $\tilde{e}(\{a\})$, where 
$$
p_n:=\sum_{i=1}^{k_n}\lambda_{n,i} \left(\tilde{e}(\{a\})+\tilde{e}(B_{n,i})\right).
$$
Set also $B_0:=\bigcup_{n\ge 1}\bigcup_{i=1}^{k_n}B_{n,i}$. 
Recalling that $\lim_n \langle p_n,u\rangle =\langle \tilde{e}(\{a\}),u\rangle$ for each $u \in \mathbf{R}^Z$ and choosing $u=e_a$, we obtain $\lim_n \sum_{i=1}^{k_n}\lambda_{n,i}=1$. First, suppose that $|B_0|<\infty$ and write each $p_n$ as $\sum_{\emptyset\neq B\subseteq B_0}\lambda_{n,B} \left(\tilde{e}(\{a\})+\tilde{e}(B)\right)$, where each $\lambda_{n,B}$ is possibly $0$ and at least one of them is nonzero. Letting $u$ be the characteristic function of $Z_1$, we obtain 
\begin{displaymath}
\begin{split}
1=\langle \tilde{e}(\{a\}), u\rangle
&=\lim_{n\to \infty}\langle \sum_{\emptyset \neq B\subseteq B_0}\lambda_{n,B}\left(\tilde{e}(\{a\})+\tilde{e}(B)\right),u\rangle \\ 
&=\lim_{n\to \infty} \sum_{\emptyset \neq B\subseteq B_0}\lambda_{n,B}\left(1+\frac{1}{|B|}\right).
\end{split}
\end{displaymath}
Since $\lim_n\sum_{\emptyset \neq B\subseteq B_0}\lambda_{n,B}=1$, we obtain that 
$$
\lim_{n\to \infty} \sum_{\emptyset \neq B\subseteq B_0}\frac{\lambda_{n,B}}{|B|}=0. 
$$
However, this is impossible since 
$$
\liminf_{n\to \infty} \sum_{\emptyset \neq B\subseteq B_0}\frac{\lambda_{n,B}}{|B|} \ge 
\liminf_{n\to \infty} \sum_{\emptyset \neq B\subseteq B_0}\frac{\lambda_{n,B}}{|B_0|}=\frac{1}{|B_0|}>0.
$$
This contradiction proves that $B_0$ is an infinite set. Hence, setting $z_n:=\max \bigcup_{i=1}^{k_n}B_{n,i}$ for each $n\ge 1$, we obtain that there exists a strictly increasing subsequence $(z_{n_m})_{m\ge 1}$ with the property that $z_{n_m}\notin \bigcup_{1\le t<n_m}\bigcup_{i=1}^{k_t}B_{t,i}$ for each $m\ge 1$. Finally, let $u_0 \in \mathbf{R}^Z$ be the utility function supported on $\{z_{n_t}: t\ge 1\}$ defined by 
$$
u_0(z_{n_m}):=m \cdot \frac{\max\{|B_{n_m,i}|^2: i \in [1,k_{n_m}]\}}{\min\{\lambda_{n_m,i}: i \in [1,k_{n_m}]\}}
$$
for all $m\ge 1$. It follows by construction that 
\begin{displaymath}
\begin{split}
\langle p_{n_m},u_0\rangle&\ge p_{n_m}(z_{n_m})u_0(z_{n_m})\\
&= u_0(z_{n_m}) \sum_{i=1}^{k_{n_m}}\lambda_{n_m,i} \left(\tilde{e}(\{a\})+\tilde{e}(B_{n_m,i})\right)(z_{n_m})\\
&\ge u_0(z_{n_m}) \sum_{i=1}^{k_{n_m}}\lambda_{n_m,i} \tilde{e}(B_{n_m,i})(z_{n_m})\\
&\ge u_0(z_{n_m}) \cdot \frac{\min\{\lambda_{n_m,i}: i \in [1,k_{n_m}]\}}{\max\{|B_{n_m,i}|^2: i \in [1,k_{n_m}]\}}=m,
\end{split}
\end{displaymath}
for all $m\ge 1$. This shows that the subsequence $\left(\langle p_{n_m}, u_0\rangle \right)_{m\ge 1}$ cannot be convergent to $\langle \tilde{e}(\{a\}), u_0\rangle$, contradicting the hypothesis that $(p_n)_{n\ge 1}$ is weak-convergent to $\tilde{e}(\{a\})$. Therefore $\tilde{e}(\{a\})\notin C$. 

\medskip

Lastly, let us show that $\tilde{e}(\{a\})$ belongs to the weak-closure $\overline{C_0}$ of $C_0$. To this aim, suppose for the sake of contradiction that $\tilde{e}(\{a\})\notin \overline{C_0}$. Thanks to the 
Strong Separating Hyperplane Theorem, 
see e.g. \cite[Theorem 8.17]{MR2317344}, there exists a linear functional $f: \mathbf{R}_0^Z \to \mathbf{R}$ such that 
$$
f(\tilde{e}(\{a\}))=-1
\quad \text{ and }\quad 
f(p) \ge 0 \text{ for all }p \in \overline{C_0}.
$$
Now, since $Z$ is uncountable, there exists a positive integer $k_0$ and an uncountable subset $\tilde{Z}_1\subseteq Z_1\setminus \{a\}$ such that 
\begin{equation}\label{eq:uncountable}
\forall b \in \tilde{Z}_1, \quad f(\tilde{e}(\{a\})+\tilde{e}(\{b\})) \le k_0,
\end{equation}
Note that, since $\tilde{e}(\{a\})+\tilde{e}(B)$ belongs to $C_0$ for each $B \in \mathscr{B}$, then $f(\tilde{e}(B)) \ge 
1$. 
To conclude, let $B$ be a finite subset of $\tilde{Z}_1$ such that $|B|=k_0+2$. It follows that 
\begin{displaymath}
\begin{split}
0&\le f(\tilde{e}(\{a\})+\tilde{e}(B))\\
&=f\left(\frac{1}{|B|}\sum_{b \in B}(\tilde{e}(\{a\})+\tilde{e}(\{b\}))+\tilde{e}(B)-\frac{1}{|B|}\sum_{b \in B}\tilde{e}(\{b\})\right)\\
&=\frac{1}{|B|}\sum_{b \in B}f\left(\tilde{e}(\{a\})+\tilde{e}(\{b\})\right)+\left(\frac{1}{|B|^2}-\frac{1}{|B|}\right)\sum_{b \in B}f\left(\tilde{e}(\{b\})\right)\\
&\le \frac{k_0}{|B|}+|B|\left(\frac{1}{|B|^2}-\frac{1}{|B|}\right)
<0.
\end{split}
\end{displaymath}
This contradiction proves that $\tilde{e}(\{a\})\in \overline{C_0}$, completing the proof. 
\end{proof}
Notice that the hypothesis that $Z$ is an uncountable set has been used in the above construction only to prove the existence of the infinite set $\tilde{Z}_1$ satisfying \eqref{eq:uncountable}. An analogue technique has been used by the author, in a different context, in the proof of \cite[Theorem 1.2]{MR3836186}. 

\medskip

\begin{proof}[Proof of Corollary \ref{cor:noncontradiction}] 
Suppose by contradiction that $\succsim^\prime$ is a $C(Z)$-sequentially continuous on $\mathscr{P}$ which satisfies the independence axiom and
$$
p \succsim q \text{ if and only if }p \succsim^\prime q
$$
for all lotteries $p,q \in \Delta$.  
Thanks to \cite{Dubra_et_al}, $\succsim^\prime$ has a representation as in \eqref{eq:claimedmultiutility} by some set $U\subseteq C(Z)$. 
Thus $\succsim$ is $C(Z)$-sequentially continuous, hence also $w$-sequentially continuous. It follows that $\succsim$ has a representation as in \eqref{eq:claimedmultiutility} with the same set  $U$ (which is contained, of course, in $\mathbf{R}^Z$). This contradicts Theorem \ref{thm:mainnegativecharacterization}. 
\end{proof}

\medskip

\begin{proof}
[Proof of Theorem \ref{thm:firstrepresentation}]
The \textsc{If} part is straighforward, hence we show the \textsc{Only If} part. Choosing the weak-topology $w=\sigma(\mathbf{R}_0^Z, \mathbf{R}^Z)$, consider the duality between lotteries and utilities $\langle \mathbf{R}_0^Z, \mathbf{R}^Z\rangle$, with duality map given by 
$$
\forall p \in \mathbf{R}_0^Z, \forall u \in \mathbf{R}^Z \quad 
\langle p,u\rangle :=
\mathbf{E}_p[u],
$$
At this point, define the pointed cone $C$ as in \eqref{eq:definitionC}. 
Fix two lotteries $p,q \in \Delta$. Since $\succsim$ is a reflexive binary relation satisfying the independence axiom, it is well known and easy to check that $p\succsim q$ if and only if $p-q \in C$.  
Thanks to \cite[Theorem 2.2]{LeoPrin22}, there exists a nonempty family $\mathscr{U}\subseteq \mathcal{P}(\mathbf{R}^Z)$ of nonempty open sets such that $p-q \in C$ if and only if 
$$
\forall U \in \mathscr{U}, \exists u \in U, \quad 
\langle u,p-q\rangle \ge 0.
$$
The conclusion follows by the linearity of the duality map. 
\end{proof}


\begin{rmk}\label{rmk:uniquenessgeneral}
As it follows by \cite[Theorem 2.5]{LeoPrin22}, if $\mathscr{V}\subseteq \mathcal{P}(\mathbf{R}^Z)$ is another family of nonempty sets satisfying the representation given in Theorem \ref{thm:firstrepresentation}, then 
$$
\{G \in \mathscr{G}: U\subseteq G \text{ for some }U \in \mathscr{U}\}
=\{G \in \mathscr{G}: V\subseteq G \text{ for some }V \in \mathscr{U}\}.
$$
Here, $\mathscr{G}$ stands for the family of nonempty sets $G$ of the type 
$$
\left\{u \in \mathbf{R}^Z: \sum\nolimits_{z \in Z}u(z)p(z)<0\right\},
$$
for some finitely-supported nonzero function $p \in \mathbf{R}_0^Z$. 
Note that each $G \in \mathscr{G}$ is a weak-open set. 
A related uniqueness result can be found in \cite[Appendix A.3]{MR3957335}. 
\end{rmk}
\begin{rmk}
Of course, since Theorem \ref{thm:firstrepresentation} is not related to any continuity axiom on the binary relation $\precsim$, one can choose also a different locally convex topology on $\mathbf{R}_0^Z$. For instance, choosing the weak-topology $\sigma(\mathbf{R}_0^Z, \mathbf{R}_0^Z)$, one obtains that there exists a nonempty family $\mathscr{U}\subseteq \mathcal{P}(\mathbf{R}_0^Z)$ of nonempty sets of finitely-supported utility functions such that 
\eqref{eq:firstcharacterizationtoogeneral} 
holds 
for all lotteries $p,q \in \Delta$. 
\end{rmk}

\section{Concluding Remarks}
The main message of this work is that, given a nonempty outcome set $Z$, every binary relation $\succsim$ on $\Delta$ is a $w$-sequentially continuous preorder satisfying the independence axiom can be represented as in \eqref{eq:claimedmultiutility} for all lotteries $p,q \in \Delta$, for some family $U$ of utility functions, if and only if $Z$ is countable. 

On the other hand, Theorem \ref{thm:mainintermediate} suggests that stronger continuity requirements on $\succsim$ may lead to characterizations which are independent of the cardinality of $Z$: is it true that, replacing $w$-sequential continuity with $w$-continuity, the analogue of Theorem \ref{thm:mainweaktopology} holds for every $Z$?

Lastly, it would be also interesting to understand how to compare different representation of the same preference over lotteries. More precisely, suppose that $W_1,W_2\subseteq \mathbf{R}^Z$ are vector spaces of utilities with the \textup{(}\textsc{L}\textup{)}-property, and let $\succsim$ be a preorder over $\Delta$ which satisfies the independence axiom. Is it true that $\succsim$ is both $W_1$[-sequentially] continuous and $W_2$[-sequentially] continuous if and only if it is $(W_1\cap W_2)$[-sequentially] continuous?

\subsection{Acknowledgments}

I am grateful to PRIN 2017 (grant 2017CY2NCA) for financial support, and to Fabio Maccheroni, Massimo Marinacci, Simone Cerreia--Vioglio (Universit\'{a} Bocconi), Giulio Principi (New York University), and Jochen Wengenroth (Trier University) for useful comments. 

\bibliographystyle{amsplain}

\end{document}